\documentclass[12pt]{amsart}
\usepackage{amsfonts,amssymb}
\usepackage[all]{xypic}
\usepackage{amsmath}
\usepackage{amsthm,amscd,latexsym}
\usepackage{color}

\newtheorem{theorem}{Theorem}[section]
\newtheorem{corollary}[theorem]{Corollary}
\newtheorem{Definition}[theorem]{Definition}
\newtheorem{lemma}[theorem]{Lemma}
\newtheorem{proposition}[theorem]{Proposition}
\newtheorem{Example}[theorem]{Example}
\newtheorem{Remark}[theorem]{Remark}

\newenvironment{remark}{\begin{Remark}\begin{em}}{\end{em}\end{Remark}}
\newenvironment{example}{\begin{Example}\begin{em}}{\end{em}\end{Example}}
\newenvironment{definition}{\begin{Definition}\begin{em}}{\end{em}\end{Definition}}

\newcommand{\De}{\mathcal D}

\newcommand{\ve}{\varepsilon}

\begin{document}
\title{An Inverse Function Theorem Converse}
\author[Lawson]{Jimmie Lawson}
\address{Department of Mathematics, Louisiana State University,
Baton Rouge, LA 70803, USA}\email{lawson@math.lsu.edu}
\date{December, 2018}

\maketitle
\begin{abstract}
We establish the following converse of the well-known inverse function theorem.  Let $g:U\to V$ and 
$f:V\to U$ be inverse homeomorphisms between open subsets of Banach spaces.  If $g$ is differentiable of class
$C^p$ and $f$ if locally Lipschitz, then the Fr\'echet derivative of $g$ at each point of $U$ is invertible and
$f$ must be differentiable of class $C^p$.
\end{abstract}

\noindent Primary 58C20; Secondary 46B07,  46T20, 46G05, 58C25


\section{Introduction}  
A general form of the well-known inverse function theorem asserts that if $g$ is a differentiable function of class $C^p$, $p\geq 1$, between two
open subsets of Banach spaces and if the Fr\'echet derivative of $g$ at some point $x$ is invertible, then locally around $x$, there exists a 
differentiable inverse map $f$ of $g$ that is also of class $C^p$.    But in various settings, one may have the inverse function $f$ readily at
hand and want to know about the invertibility of the Fr\'echet derivative of $g$ at $x$ and whether $f$ is of class $C^p$.  Our purpose in this 
paper is to present a relatively elementary proof of this converse result under the general hypothesis that the inverse $f$ is
(locally) Lipschitz. Simple examples like $g(x)=x^3$ at $x=0$ on the real line show that the assumption of continuity alone is not enough.
Thus it is a bit surprising that the mild strengthening to the assumption that the inverse is locally Lipschitz suffices.  

Helpful tools for the task at hand have been developed in the intense study of Lipschitz functions in the Banach space setting
motivated by Rademacher's theorem concerning the existence of an abundance of points of differentiability  of Lipschitz mappings
in the setting of euclidean spaces.  In particular we recall in the next section a useful generalization of the chain rule by O.\ Maleva and D.\ Preiss \cite{MP}.
In Sections 3 and 4 we present our main results and provide some follow-up illustrative material in Section 5.

For a comprehensive reference
to the  inverse function and implicit function theorems and related theory, we refer the reader to \cite{KP}.

\section{A general version of the chain rule}
In this section we recall some notions of differentiability of Lipschitz functions between open subsets of a Banach spaces and a generalized 
chain rule from  the work of Maleva and Preiss \cite{MP}. This material will be crucial to the derivation of our main results in the next section.

Suppose that $Y$ and $Z$ are Banach spaces, $U$ is a nonempty open subset of $Y$,  $f : U \to Z$, and $y\in U$, $v \in Y $.
Recall that  $\lim_{t\to 0^+} (f(y+tv)-f(y))/t$, if it exists, is called the \emph{one-sided directional derivative of $f$ at $y$ in the direction}
$v$.  Similarly if $\lim_{t\to 0} (f(y+tv)-f(y))/t$ exists, it is called the (\emph{bilateral}) \emph{directional derivative of $f$ at $y$ in the direction} $v$.  If the
directional derivative of $f$ at $y$ in the direction $v$ exists for all $v\in Y$, then the mapping from $Y$ to $Z$ sending $v$ to its directional
derivative is, by definition, the \emph{G\^{a}teaux derivative} (by some authors the mapping is also required to be a continuous linear map).

 Maleva and Preiss \cite{MP} have given the following generalization of the one-sided directional derivative.
\begin{definition}
The \emph{derived set of $f$ at the point $y$ in the direction of $v$} is defined as the set
$\De f(y, v)$ consisting  of all existing limits $\mathrm{lim}_{n\to\infty}(f(y + t_nv)-f(y))/t_n$, where $t_n\searrow 0$. The
\emph{$\delta$-approximating derived set of $f$ at $y$ in the direction of $v$} is defined, for $\delta>0$, by
\begin{equation*}
\De_\delta f(y,v)=\Big\{ \frac{f(y + tv)-f(y)}{t}: 0<t <\delta\Big\}.
\end{equation*}
\end{definition}

\begin{remark}
It is easy to see that
\begin{equation}\label{E:derived}
\De f(y,v)=\bigcap_{\delta>0} \overline{\De_\delta f(y,v)}
\end{equation}
In general the derived set may be empty, a single point, or multi-valued.  If 
$\De f(y,v)$ is a single point, then it is the one-sided directional derivative 
of $f$ at $y$ in the direction $v$,  and we denote it by $f_+'(y,v)$. If the directional derivative of $f$ at $y$
in the direction $v$ exists, it is denoted $f'(y,v)$.  

If $f$ is G\^{a}teaux differentiable at $y$, then the G\^{a}teaux derivative at $y$ is equal
to $f'(y,\cdot)=\De f(y,\cdot)$.  
\end{remark}

We recall a general version of the chain rule from \cite[Corollary 2.6]{MP}.
\begin{proposition}\label{P:chain}
Suppose $X$, $Y$ are Banach spaces, $x\in U$, an open subset of $X$, and $V$ is an open subset of $Y$.
If $g :U\to V$ is continuous and has a one-sided directional derivative at $x$ in the
direction of $v$  and  $f :V\to  Z$ is Lipschitz, then
$$\De(f \circ g)(x, v) = \De f(g(x), g_+'(x,v)).$$
\end{proposition}

\begin{proof} Let $a\in \De(f \circ g)(x, v)$.  Then there exists a sequence $t_n\to 0^+$ such that 
$$a=\lim_n \frac{f\circ g(x+t_nv)-f\circ g(x)}{t_n}.$$
Let $\kappa$  be a Lipschitz constant for $f$ on $V$. Let $\ve>0$ and choose $N$ such that
$\Vert(g(x+t_nv)-g(x))/t_n-g_+'(x,v)\Vert<\ve/\kappa$ for $n\geq N$.
We then note for $n\geq N$
\begin{eqnarray*}
\bigg\Vert  \frac{f\circ g(x+t_nv)-f\circ g(x)}{t_n}\!\!&-& \!\!\frac{f(g(x) +t_ng_+'(x,v))-f(g(x))}{t_n}\bigg\Vert \\
&=&  \bigg\Vert \frac{f(g(x+t_nv))-f(g(x) +t_ng_+'(x,v))}{t_n}\bigg\Vert \\
&\leq& \frac{\kappa}{t_n}\Vert g(x+t_nv)-g(x) -t_ng_+'(x,v)\Vert\\
&=&\kappa\bigg\Vert \frac{g(x+t_nv)-g(x)}{t_n} -g_+'(x,v)\bigg\Vert\leq \kappa(\ve/\kappa)=\ve.
\end{eqnarray*}
It thus follows that the sequence $[f(g(x) +t_ng_+'(x,v))-f(g(x))]/t_n$ also converges to $a$, so
$a\in \De f(g(x),g_+'(x,v))$.  The argument is reversible so the equality claimed in the proposition
holds.

\end{proof}

\section{A Converse of the Inverse Function Theorem}
The inverse function theorem asserts the existence of a local inverse if the derivative is invertible.  In this section we derive a converse result:
a  Lipschitz continuous local inverse implies an invertible derivative.

\emph{In this section we work in the following setting. Let $X,Y$ be Banach spaces, let $U$ and $V$ be nonempty open
subsets of $X$ and $Y$ resp., each equipped with the restricted metric from the containing Banach space, and let $g:U\to V$ and $f:V\to U$ be inverse homeomorphisms.}

\begin{definition} For $g:U\to V$, the \emph{G\^ateaux derivative of $g$ at} $x\in U$ is defined by $d_x^Gg(v)=g_+'(x,v)$ for all $v\in X$, provided 
such one-sided directional derivatives exist for all $v\in X$ and  the resulting map $d_x^Gg:E\to F$ is a continuous linear map.
\end{definition}

\begin{lemma}\label{L:inj}
Let $x\in U$, $y=f(x)\in V$.  Assume that $g$ has a G\^ateaux derivative $d_x^Gg: E\to F$ at $x$
and assume that $f$ is Lipschitz.  Then  $\De f(g(x),\cdot)\circ d_x^Gg$ is the identity map on $X$ and
the G\^ateaux derivative $d_x^Gg$ is injective.
\end{lemma}

\begin{proof}  Since $f\circ g$ on $U$ is the identity map, it follows directly from the definition of $\De$  that 
$\De(f\circ g)(x,\cdot)$ is the identity map on $X$. Hence by Proposition \ref{P:chain}
for any $v\in X$,
$$v=\De(f\circ g)(x,v)=\De f(g(x), g_+'(x,v))=\De f(g(x),d_x^G(v)).$$
This equality of the left-hand and right-hand sides of the equation yields the two concluding 
assertions.
\end{proof}

Lemma \ref{L:inj} yields the invertibility of the the G\^ateaux derivative in the finite dimensional setting.
\begin{corollary}  If $X,Y$ are both finite dimensional and $g$ is G\^ateaux differentiable at $x\in U$,
then $d_x^G$ is invertible.
\end{corollary}

\begin{proof}  Since $g$ and $f$ are homeomorphisms, $X$ and $Y$ must have the same dimension.  Hence the linear
map $d_x^G$ is a linear isomorphism if and only if it injective, which is the case by Lemma \ref{L:inj}.
\end{proof}
The infinite dimensional case requires more work and stronger hypotheses.
\begin{lemma}\label{L:surj}
Let $x\in U$, $y=g(x)\in V$.  Assume that $g$ is  Fr\'echet differentiable at $x$ and
that $f$ is Lipschitz on $V$.  Then the image of $X$ under  the Fr\'echet derivative $dg_x:X\to Y$ 
is dense in $Y$.
\end{lemma}

\begin{proof}
Let $M_f$ be a Lipschitz constant for $f$ on $V$. Let $w\in Y$. 
 Pick $\tau>0$ small enough so that  $y+tw\in V$ for $0<t<\tau$.  We set $z_t=\frac{f(y+tw)-f(y)}{t} $ and note from the 
 Lipschitz condition that 
\begin{equation}\label{E:lips}
\Vert z_t\Vert=\Big\Vert  \frac{f(y+tw)-f(y)}{t}\Big\Vert\leq \frac{1}{t} M_f\Vert (y+tw)-y\Vert=M_f\Vert w\Vert.
\end{equation}
We divide the remainder of the proof into steps.\\
\emph{Step 1}: $(g(x+tz_t)-g(x))/t=w$ for $0<t<\tau$.
$$g(x+tz_t)-g(x)=g\bigg(f(y)+t\Big(\frac{f(y+tw)-f(y)}{t}\Big)\bigg)-y=gf(y+tw)-y=tw,$$
so $(g(x+tz_t)-g(x))/t=w$.  \\
\emph{Step 2}:  \emph{For $\ve>0$ there exists $t<\tau$ such that 
$\Vert (g(x+tz_t)-g(x))/t-dg_x(z_t)\Vert<\ve$.} 
From the Fr\'echet differentiability of $g$ at $x$, the Fr\'echet derivative $dg_x$ satisfies
$$\lim_{u\to 0} \frac{\Vert g(x+u)-g(x)-dg_x(u)\Vert}{\Vert u\Vert}=0.$$
For $\ve >0$ pick $\delta> 0$ such that
$$\frac{\Vert g(x+u)-g(x)-dg_x(u)\Vert}{\Vert u\Vert}<\frac{\ve}{M_f\Vert w\Vert} \mbox{ whenever } 0<\Vert u\Vert<\delta.$$
 Pick  $t>0$ such that $y+tw\in V$ and $t M_f\Vert w\Vert< \delta$. 
We conclude from inequality (\ref{E:lips})  and the preceding that
 \begin{eqnarray*}
 \Big\Vert\frac{g(x+tz_t)-g(x))-dg_x(tz_t)}{t}\Big\Vert&=&\Vert z_t\Vert\Big\Vert\frac{g(x+tz_t)-g(x))-dg_x(tz_t)}{\Vert tz_t\Vert}\Big\Vert\\
 &<& M_f\Vert w\Vert \frac{\ve}{M_f\Vert w\Vert} =\ve
 \end{eqnarray*}
 Using the linearity of of the Fr\'echet derivative $dg_x$, we obtain $dg_x(tz_t)=tdg_x(z_t)$, which allows us to
 rewrite the first entry in the preceding string to obtain
 $$ \Big\Vert\frac{g(x+tz_t)-g(x)}{t}-dg_x(z_t)\Big\Vert <\ve,$$
 which establishes the Step 2.
 
 We note that combining Steps 1  and 2 yields
 $$ \Vert w-dg_x(tz_t)\Vert\leq \Big\Vert w-\frac{g(x+tz_t)-g(x)}{t}\Big\Vert+
\Big\Vert \frac{g(x+tz_t)-g(x)}{t}-dg_x(z_t)\Big\Vert\leq \ve.$$
 Since $w\in Y$ and $\ve>0$ were chosen arbitrarily, this completes the proof.
 
\end{proof}

We come now to a central result of the paper, what we are calling a converse
of the inverse function theorem.
\begin{theorem}\label{T:main}
Let $X$ and $Y$ be Banach spaces, let $U$ and $V$ be nonempty open
subsets of $X$ and $Y$ resp., and let $g:U\to V$ and $f:V\to U$ be inverse
 homeomorphisms. Let $x\in U$ and $y=g(x)\in  V$.  Let $g$ be 
 Fr\'echet differentiable at $x$ and let $f$ be Lipschitz continuous on $V$.  
 Then $dg_x(\cdot): X\to Y$  is an  isomorphism.
 \end{theorem}
 
 \begin{proof}  By Lemma \ref{L:inj}  the Fr\'echet derivative $dg_x$ is
 injective, hence a linear isomorphism onto its image $Z=g(X)$, a subspace of $Y$, 
 and has inverse  $\De f(y,\cdot):Z\to X$.  (In particular in this setting for $w\in Z$
 it must be the case that $\De f(y,w)$ is a singleton, which we could write
 alternatively as $f_+'(y,w)$.)
 
 Let $M_f$ be the Lipschitz constant for $f$ on $V$.  By equation (\ref{E:lips})
 every member of $\De_\delta f(y,w)$ is bounded in norm by $M_f\Vert w\Vert$, and hence
 the same is true for $\De f(y,w)=\bigcap_{\delta>0} \overline{\De_\delta f(y,w)}$.  
 We conclude that $\De f(y,\cdot):Z\to X$ is Lipschitz for the Lipschitz constant
 $M_f$.  

 By Lemma \ref{L:surj} $Z$ is dense  in $Y$.    Thus the linear Lipschitz map 
   $\De f(y,\cdot):Z\to X$ extends uniquely to a linear Lipschitz map
  (hence a bounded linear operator) from $Y$ to $X$.  Label
  $dg_x=\Gamma_X$, $\De f(y,\cdot)=\Gamma_Z$ and 
  its extension to $Y$ by $\Gamma_Y$.  We note for $0\ne z\in Z$,
  $$\Vert z\Vert =\Vert \Gamma_X(\Gamma_Z(z)\Vert\leq \Vert \Gamma_X\Vert \Vert \Gamma _Z(z)\Vert,$$
  so $(1/\Vert\Gamma_X\Vert)\Vert z\Vert \leq \Vert \Gamma _Z(z)\Vert$.
  By continuity of the norm and density of $Z$ in $Y$ this inequality carries over from
  $\Gamma_Z$ to its extension  $\Gamma_Y$.    But this extended inequality implies
  $\Gamma_Y$ has a trivial kernel, which means that $\Gamma_Y$ is injective.   But if 
  $Z$ is proper in $Y$, then this is impossible, since $\Gamma_Z(Z)=X$.  Hence 
  $Z=Y$ and $\Gamma_Z=\Gamma_Y$ is an inverse for $dg_x$.
   \end{proof}
   
  \section{A Global Result}
   We can use Theorem \ref{T:main} to derive useful results for studying differentiable functions
   with locally Lipschitz inverses, in particular for deriving differentiability properties of the inverse.
     In the following $C^p$ means have continuous derivatives through
   order $p$ for $p$ a positive integer, have continuous derivatives of all order for $p=\infty$,
   and being analytic (having locally power series expansions) for $p=\omega$.
  
  \begin{theorem}\label{T:ift}
  Let $X$ and $Y$ be Banach spaces, let $U$ and $V$ be nonempty open
subsets of $X$ and $Y$ resp., and let $g:U\to V$ and $f:V\to U$ be inverse
 homeomorphisms.  Assume further that $g$ is of class $C^p$ on $U$ 
 for some $p\geq 1$ and that $f$ is locally Lipschitz on $V$.  
 Then $g$ and $f$ are inverse diffeomorphisms of class $C^p$.
 \end{theorem}
 
 \begin{proof} We fix $x\in U$ and $y=f(x)\in V$.  
 We apply Theorem \ref{T:main} to small enough neighborhoods of $x$ and $y$
 so that $f$ is Lipschitz to see that the hypotheses of 
 the standard inverse function theorem (in the Banach space setting)  are satisfied.
 The conclusions of this theorem then follow locally from the inverse function theorem.
 (See, for example, \cite[Theorem 1.23]{Up} for the $C^\omega$-case.) 
 In this way we obtain that the conclusions of the theorem hold locally for each
 $x\in X$ and hence hold globally.
 \end{proof}
 
 We remark that the preceding results can be generalized to the setting of Banach manifolds.
 In order to obtain the preceding results in this setting one needs Lipschitz charts (between the manifold
 metric and  the Banach space metric)  at each point and then the preceding results readily extend to
 this more general setting.  
 
  \section{An Application}
  We sketch in this section one setting in which our general converse of the inverse function theorem
  can be fruitfully applied and briefly consider a specific example.
  
  Suppose we are given some equation of the form $F(x,y)=0$, where $x,y$ belong to some
  given open subset of a Banach space.  In some cases it might be possible to solve the equation
  for $y$ in terms of $x$, i.e., $y=g(x)$, which can seen to be of class $C^p$.  Let's suppose further
  that $g$ has an inverse given by $x=f(y)$, but no corresponding explicit description of this function.
  If it can be shown, however, that $f$ is (locally) Lipschitz, then we use the results of the preceding section
  to show that $f$ is also of class $C^p$.  
  
  \begin{example} Let $\mathcal{B}(H)$ be the $C^*$-algebra of bounded linear operators on the Hilbert
  space $H$.  We consider the Banach space $\mathbb{S}$ of hermitian operators and the open cone
  of positive invertible hermitian operators $\mathbb{P}$.   In recent years a useful notion of a multi-variable
  geometric mean $\Lambda$ on $\mathbb{P}$ has arisen \cite{LL13}, \cite{LL14}, generally called the Karcher mean.  One useful 
  characterization of this $n$-variable mean  $\Lambda(A_1,\ldots,A_n)$ for $A_1,\ldots,A_n\in\mathbb{P}$
  is that it is the unique solution $X$ of the equation
  $$\log (X^{-1/2}A_1X^{-1/2})+\cdots +\log (X^{-1/2}A_nX^{-1/2})=0.$$
  If we fix $A_1,\ldots, A_{n-1}$, let $Y=A_n$, and the left-hand side of the equation be $F(X,Y)$,
  then we are in the general setting of the previous paragraph with $U=V=\mathbb{P}$. One can rather easily solve
  $F(X,Y)$ for $Y$ and see it is an analytic function $g$ of $X$, but not conversely.  
  However the inverse $f$ is given by $f(Y)=\Lambda(A_1,\ldots, A_{n-1}, Y)$.  
The (local) Lipschitz property is a basic property of $\Lambda$, so the earlier
results yield $X$, the geometric mean, as an analytic function of $Y$, or alternatively we
can say $\Lambda$ is an analytic function of each of its variables.  This 
turns out to be an important property of $\Lambda$.
This type of analysis can be applied to other important operator means.
    \end{example}

\end{document}